\documentclass{amsart}[12pt]

\usepackage{yfonts} %
\usepackage{amssymb} %
\usepackage{amsthm}
\usepackage{array}
\usepackage{booktabs}%
\usepackage{hhline}%
\usepackage{xy} %
\usepackage{epsfig}%
\usepackage{color}%
\usepackage{upgreek}
\usepackage[english]{babel}
\usepackage{epigraph}%
\usepackage{fancybox}%
\usepackage{shadow}
\usepackage{afterpage}
\usepackage{mathrsfs}
\usepackage{enumitem}
\usepackage{subcaption}
\usepackage{graphicx}

\newtheorem{theorem}{Theorem}

\newtheorem{proposition}{Proposition}
\theoremstyle{definition}
\newtheorem{definition}{Definition}
\newtheorem{remark}{Remark}

\theoremstyle{plain}
\newtheorem{corollary}{Corollary}

\newtheorem*{jt}{Jonker's Theorem}
\newtheorem*{sacks}{Sacksteder Rigidity Theorem}

\newcommand{\vt}{\vspace{.1cm}}
\newcommand{\vtt}{\vspace{.2cm}}
\newcommand{\R}{\mathbb{R} }

\newcommand{\h}{\mathbb{H} }

\begin{document}
\title[Isometric Immersions into Space Forms]{Convexity, Rigidity, and Reduction of  \\ Codimension of Isometric Immersions\\ into Space Forms}
\author[ R. F. de Lima and R. L. de Andrade]{Ronaldo F. de Lima and Rubens L. de Andrade}
\address{Departamento de Matem\'atica - Universidade Federal do Rio Grande do Norte}
\email{rubensleao@ccet.ufrn.br, ronaldo@ccet.ufrn.br}
\subjclass[2010]{53B02 (primary),  53C42 (secondary).}
\keywords{isometric immersion  -- convexity -- rigidity -- reduction of codimension.}
\maketitle
\centerline{To Elon Lima, \emph{in memoriam},}
\centerline{and}
\centerline{Manfredo do Carmo, on his 90th birthday}
\begin{abstract}
We consider isometric immersions of  complete connected Riemannian manifolds  into
space forms of nonzero constant curvature.
We prove that if such an immersion is compact and has semi-definite second fundamental form,
then it is an embedding with codimension one, its image bounds a convex  set, and it is rigid.
This result generalizes previous ones by M. do Carmo and E. Lima, as well as by M. do Carmo and F. Warner.
It also settles affirmatively a conjecture by do Carmo and Warner. We establish a similar result
for complete isometric immersions satisfying a stronger
condition on the second fundamental form. We extend  to the context of
isometric immersions in space forms  a classical theorem for Euclidean hypersurfaces due to Hadamard. In this same context,
we prove an existence theorem for hypersurfaces with prescribed boundary and vanishing Gauss-Kronecker curvature.
Finally, we show that isometric immersions into space forms which are regular outside the set
of totally geodesic points admit a reduction of codimension to one.
\end{abstract}

\section{Introduction}
Convexity and rigidity are among the most essential concepts in the theory of submanifolds.
In his work, R. Sacksteder established
two fundamental results involving these concepts. 
Combined, they state that for \,$M^n$\,  a non-flat
complete connected Riemannian manifold with nonnegative  sectional curvatures, any isometric immersion \,$f:M^n\rightarrow\R^{n+1}$\,
is, in fact, an embedding and has \,$f(M)$\, as the boundary of a convex set. In particular, \,$M$\,
is diffeomorphic to the Euclidean space \,$\R^n$\,  or to the unit sphere \,$S^n$\, (Sacksteder \cite{sacksteder1}).
In the latter case, \,$f$\, is rigid, that is, for any other isometric immersion  \,$g:M^n\rightarrow \R^{n+1},$\, there exists a rigid motion
\,$\Phi:\R^{n+1}\rightarrow \R^{n+1}$\, such that \,$g=\Phi\circ f$\, (Sacksteder \cite{sacksteder2}).

The convexity part of this statement 
is a Hadamard-Stoker type
theorem, since it refers to the results of Hadamard \cite{hadamard} and Stoker \cite{stoker}, who considered
the compact and complete cases, respectively, assuming \,$n=2$\, and \,$M$\, with positive curvature everywhere.
The rigidity part is a generalization of the classical Cohn-Vossen rigidity theorem for ovaloids.

Naturally, the extension of Sacksteder's results to
general isometric immersions into space forms became a matter of interest.
However, as shown by the standard  immersion of \,$S^n\times S^n$\,  into
\,$\R^{2n+2},$\, Sacksteder Theorem \cite{sacksteder1} is not valid in higher codimension.
On the other hand, its proof relies mostly on the semi-definiteness of the
second fundamental form of the immersion, which, in codimension one,
is equivalent to the assumed non-negativeness of the sectional curvatures of \,$M.$\,
Thus, in order to get similar results in higher codimension, it is natural
to assume that the second fundamental form is semi-definite, that is,
at each point and in any normal direction, all
the nonzero eigenvalues of the corresponding shape
operator have the same sign.

M. do Carmo and E. Lima \cite{docarmo-lima} established  a Hadamard-Stoker type
theorem  for isometric immersions  of compact manifolds into Euclidean space with
arbitrary codimension.
Namely, they proved that if \,$M^n$\, is a compact connected Riemannian manifold and
\,$f:M^n\rightarrow\R^{n+p}$\, is an isometric immersion whose  second fundamental form is semi-definite
(and definite at one point), then
\,$f$\, admits a reduction of  codimension to one and embeds \,$M$\, onto the boundary of a
compact convex set.
Subsequently, this result was extended by L. Jonker \cite{jonker} to  complete isometric
immersions (see Section \ref{sec-proof1and2} for a precise statement).

In \cite{docarmo-warner}, M. do Carmo and F. Warner considered compact connected
hypersurfaces \,$f:M^n\rightarrow S^{n+1}.$\,
They proved that if  all sectional curvatures of  \,$M$\,
are greater than or equal to \,$1$\, (the curvature of the ambient space), then \,$f$\, is rigid and embeds
\,$M$\, onto the boundary of a compact convex set contained in an open hemisphere of \,$S^{n+1}.$\,
In addition, it was shown that, except for the rigidity part, this theorem remains valid
if one replaces the sphere \,$S^{n+1}$\, by the hyperbolic space \,$\h^{n+1}$\, and assume that
all sectional curvatures of  \,$M$\, are no less than \,$-1.$\,
The authors also conjectured the rigidity of \,$f$\, for this case.

In the present paper, we extend  do Carmo-Lima and do Carmo-Warner theorems to isometric
immersions of arbitrary codimension into space forms
of nonzero constant curvature.
We settle affirmatively, as well, the aforementioned do Carmo and Warner's conjecture.
More precisely, we obtain the following  result.

\begin{theorem}  \label{th-convexity}
Let \,$f:M^n\rightarrow Q_c^{n+p}$\, be an
isometric immersion of a  compact connected Riemannian manifold  into
the space form of constant curvature \,$c\ne 0.$\, Assume  that \,$f$\, is non-totally geodesic and has
semi-definite second fundamental form.
Then, \,$f$\, is an embedding of  \,$M$\, into a totally geodesic $(n+1)$-dimensional
submanifold \,$Q_c^{n+1}\subset Q_c^{n+p},$\,
\,$f(M)$\, is the boundary
of a compact convex set of \,$Q_c^{n+1},$\, and \,$f$\ is rigid.
In particular, \,$M$\, is diffeomorphic to a sphere.
\end{theorem}

As is well known, the flat $n$-dimensional Clifford torus can  be embedded
into the hyperbolic space \,$\h^{2n}.$\,  Additionally, one can easily
obtain non-totally geodesic isometric immersions \,$f:S^n\rightarrow S^{2n+1}$\, whose codimension
cannot be reduced to one (see \cite{dajczer}, pg 76). So,  Theorem \ref{th-convexity} is no longer true if we
replace the condition on the second fundamental form by that of \,$M$\, having no sectional curvatures
less than \,$c.$\, However, \,$M$\, has this latter property if the immersion
\,$f:M^n\rightarrow Q_c^{n+p}$\, has semi-definite second fundamental form (see Proposition \ref{prop-sdsr}).

Due to the  Bonnet-Myers Theorem and the above considerations, we can  replace compactness by completeness
in the statement of Theorem \ref{th-convexity} if \,$c>0.$\,
Nevertheless, for \,$c<0,$\, we cannot expect to obtain a Hadamard-Stoker type theorem if we assume that
\,$M$\, is complete and noncompact. Indeed, there are complete immersions  in hyperbolic space which are not embeddings and whose
second fundamental form is semi-definite (see \cite{spivak}, pg 124).
Thus, in this context, to ensure that the immersion is actually an embedding, we
need stronger conditions on the second fundamental form.

R. J. Currier \cite{currier} obtained a Hadamard-Stoker type theorem for complete hypersurfaces in hyperbolic space which are
locally supported by horospheres,
that is,  the eigenvalues of their shape operators are all greater than or equal to \,$1.$\,
Here, we consider the analogous problem in arbitrary codimension and obtain the following result.

\begin{theorem} \label{th-hconvexity}
Let \,$f:M^n\rightarrow\h^{n+p}$\, be an isometric immersion
of an orientable complete connected Riemannian manifold \,$M^n$\, into the hyperbolic space \,$\h^{n+p}.$\,
Assume that there is an orthonormal frame \,$\{\xi_1\,, \dots ,\xi_p\}$\, in \,$TM^\perp$ such that
all the eigenvalues of the shape operators \,$A_{\xi_i}$\,  are greater than or equal to \,$1.$\,
Then, \,$f$\, admits a reduction of codimension to one, \,$f:M^n\rightarrow\h^{n+1}.$\,
As a consequence, \,$f$\, is and embedding, \,$f(M)$\, is the boundary of a convex set in \,$\h^{n+1},$\, and
\,$M$\, is either diffeomorphic to \,$S^n$\,  or to \,$\R^n.$\,
Moreover, \,$f$\, is rigid and, in fact,
\,$f(M)$\, is a horosphere of \,$\h^{n+1}$ if \,$M$\, is not compact.
\end{theorem}

Essentially, the proof of  Theorem \ref{th-convexity}  will be carried out
by means of the so called Beltrami maps, which were used
by do Carmo and Warner as well. These maps are
geodesic diffeomorphisms from either
an open hemisphere or the hyperbolic space to the
Euclidean space of same dimension.
By relying on their properties, one can  reduce certain problems
in spherical or hyperbolic spaces
to  analogous ones set in Euclidean space.

We also benefit from Beltrami maps to obtain an existence result for
hypersurfaces in space forms with prescribed boundary and vanishing Gauss-Kronecker curvature
(Corollary \ref{cor-prescribedboundary}), as well as to establish two results regarding the convex hull of bounded
domains of submanifolds of space forms (corollaries \ref{cor-interiorCHP} and \ref{cor-exteriorCHP}).
By the same token, we obtain the following  extension of a classical theorem due to Hadamard \cite{hadamard1}.

\begin{theorem}  \label{th-hypersurface}
Let \,$f:M^n\rightarrow\mathcal{H}^{n+1}$\, be a  compact connected hypersurface, where \,$\mathcal{H}^{n+1}$\, is either
the open hemisphere of \,$S^{n+1}\subset\R^{n+2}$\, centered at \,$e=(0,0,\dots ,1)$\,
or the hyperbolic space \,$\h^{n+1}\subset\mathbb{L}^{n+2}=(\R^{n+2}, \langle \,,\, \rangle_{\scriptscriptstyle{\mathbb{L}}}).$\,
Then, the following assertions are equivalent:
\begin{itemize}[parsep=1ex]
\item[\rm i)] The second fundamental form of \,$f$\, is definite everywhere.

\item[\rm ii)] The Gauss-Kronecker curvature of \,$f$\, is nowhere vanishing.

\item[\rm iii)] \,$M$\, is orientable and, for a unit normal field \,$\xi$\, defined on \,$M,$\,
the map
$$
\begin{array}{cccc}
\psi: & M^n & \rightarrow & S^n \\

      & x   & \mapsto     &  \frac{\xi(x)-\langle\xi(x),e\rangle e}{\sqrt{1-\langle\xi(x),e\rangle^2}}
  \end{array}
$$
is a well-defined diffeomorphism, where
\,$S^n$\, stands for the $n$-dimensional unit sphere of
the Euclidean orthogonal complement  of \,$e$\, in \,$\R^{n+2}.$
\end{itemize}
Furthermore, any of the above conditions implies that \,$f$\, is rigid and embeds \,$M$\, onto the
boundary of a compact convex set in \,$\mathcal{H}^{n+1}.$\,
\end{theorem}

It is easily seen that the standard minimal immersion of the two-dimensional Clifford torus
in \,$S^3$\, has non-zero Gauss-Kronecker curvature, which shows that
Theorem \ref{th-hypersurface} is not valid for general
compact hypersurfaces \,$f:M^n\rightarrow S^{n+1}.$

The second part of our paper is devoted to the problem of reducing the  codimension of certain isometric immersions
\,$f:M^n\rightarrow Q_c^{n+p},$\, which we will call $(1\,;1)$-\emph{semi-regular}. Such an immersion is  characterized
by the fact that, at each non-totally geodesic point,
its first normal space (that is, the space generated by its second fundamental form)
has constant dimension equal to \,$1.$\,  If \,$f$\, is $(1\,;1)$-{semi-regular} and
has no totally geodesic points,  it is called $(1\,;1)$-regular.

Our last result, as quoted below, extends the main results of
\cite{rodriguez-tribuzy} (theorems 1, 2, and 3) to the more general
context of $(1\,;1)$-semi-regular isometric immersions. There, the authors  studied  reduction of codimension of
$(1\,;1)$-regular isometric immersions into space forms.

\begin{theorem}  \label{th-semiregular}
Let \,$M^n$\, be a complete (compact, if \,$c\le 0$) connected Riemannian manifold, and \,$f:M^n\rightarrow Q_c^{n+p}$\,
an  $(1\,;1)$-semi-regular
isometric immersion of \,$M^n$\, into the space form \,$Q_c^{n+p}.$\,
Then, the immersion \,$f$\, admits a reduction of codimension to one
whenever  its set of totally geodesic points does not disconnect  \,$M.$\,
As a consequence, the following hold:

\begin{itemize}[parsep=1ex]
\item[{\rm i)}] $f$\, is an embedding and
\,$f(M)$\, bounds a compact convex set of \,$Q_c^{n+1},$\, provided the Ricci curvature of \,$M$\,
is nowhere less than \,$c.$

\item[{\rm ii)}] $f$\, is rigid if either \,$c>0$\, and \,$n\ge 4$\, or \,$c\le 0$\, and \,$n\ge 3.$
\end{itemize}
\end{theorem}

The paper is organized as follows. In Section \ref{sec-preliminaries}, we introduce some notation and basic results on
isometric immersions.
In Section \ref{sec-semiregular}, we discuss on semi-regular  isometric immersions and  their elementary properties.
In Section \ref{sec-beltramimaps}, we introduce the Beltrami maps and establish a result (Proposition \ref{prop-beltramimaps})
that will lead to the proofs of our theorems. In Section  \ref{sec-proof1and2}, we prove the theorems from  \ref{th-convexity} to
\ref{th-hypersurface}, and  finally, in Section \ref{sec-proof3}, we prove Theorem \ref{th-semiregular}.

\vtt

\noindent
{\bf Acknowledgments.}  We would like to acknowledge Professor
Manfredo do Carmo for his inspiring teaching and encouragement. The second author
is grateful to Marcos Dajczer, Luis Florit and Ruy Tojeiro for helpful conversations.

\section{Preliminaries} \label{sec-preliminaries}

Let us fix some notation and recall some classical results on isometric immersions which will be used throughout the paper.
For details and proofs we refer to \cite{dajczer}.

Unless otherwise stated ({\it{e.g.}}, Corollary \ref{cor-prescribedboundary}), all Riemannian manifolds we
consider here are  assumed to be \,$C^\infty$\, and of dimension \,$n\ge 2.$

We shall  denote by \,$Q_c^n$\, the $n$-dimensional space form whose sectional curvatures are constant and equal to \,$c\in\{0, 1, -1\},$\, that is,
\,$Q_0^n$\, is the Euclidean space \,$\R^n,$\, \,$Q_1^n$\, the unit sphere \,$S^n,$\, and
\,$Q_{\scriptscriptstyle{-1}}^{n}$\, is   the hyperbolic space \,$\mathbb{H}^n.$

Let \,$M^n$\, and \,$\tilde{M}^{n+p}$\, be Riemannian manifolds of dimensions \,$n\ge 2$\, and
\,$n+p>n,$\, respectively. Given an isometric immersion
$$f:M^n\rightarrow\tilde{M}^{n+p},$$
we will  write \,$TM$\, and \,$TM^\perp$\, for its tangent bundle and normal
bundle, respectively, and \,$\alpha_f:TM\times TM\rightarrow TM^\perp$\, for its second fundamental form, that is,
$$
\alpha_f(X,Y)=\widetilde{\nabla}_XY-\nabla_XY,
$$
where \,$\widetilde{\nabla}$\, and \,$\nabla$\, denote, respectively, the Riemannian connections of \,$\tilde{M}$\, and \,$M.$

Given \,$\xi\in TM^\perp,$\, we define the (self-adjoint) operator \,$A_\xi:TM\rightarrow TM$\, by
$$
A_\xi X=-(\text{tangential component of} \,\,\, \widetilde{\nabla}_X\xi)
$$
and call it the \emph{shape operator} of \,$f$\, in the normal direction \,$\xi.$\, It is easily seen that
$$
\langle A_\xi X,Y\rangle=\langle\alpha_f(X,Y),\xi\rangle \,\,\,\, \forall \,X,Y\in TM, \,\, \xi\in TM^\perp,
$$
where \,$\langle \,\,,\, \rangle$\, stands for the Riemannian metric in both \,$M$\, and \,$\tilde{M}.$

One says that the second fundamental form \,$\alpha_f$\, of \,$f$\, is \emph{semi-definite} if, for all \,$\xi\in TM^\perp,$\,
the 2-form
$$(X,Y)\in TM\times TM\mapsto \langle\alpha_f(X,Y),\xi\rangle$$
is semi-definite. Clearly, this condition
is equivalent to that of all the nonzero eigenvalues of the shape operator \,$A_\xi$\, having the same sign.

The set of totally geodesic points of an isometric immersion
\,$f:M^n\rightarrow\tilde{M}^{n+p},$\, that is, those at which the second fundamental form vanishes, will be denoted by \,$M_{\rm tot}$\,.
If, in addition,  \,$\tilde{M}$\, has constant curvature \,$c,$\, we will denote by \,$M_c$\, the set of points \,$x$\, of \,$M$\, whose sectional curvatures
\,$K_{\scriptscriptstyle{M}}(X,Y), \, X,Y\in T_xM,$\, are all equal to \,$c.$\,
Observe that, by \emph{Gauss equation}
$$
K_{\scriptscriptstyle{M}}(X,Y)=c+\langle\alpha_f(X,X),\alpha_f(Y,Y)\rangle-\|\alpha_f(X,Y)\|^2,
$$
which is valid for any orthonormal vectors \,$X, Y\in TM,$\, we have that
$$
M_{\rm tot}\subset M_c\subset M.
$$

We  define the \emph{subspace of relative nullity} of \,$f$\, at
\,$x\in M$\, as the vector subspace \,$\Delta(x)$\, of \,$T_xM$\, given by
$$
\Delta(x)=\{X\in T_xM \,;\, \alpha_f(X,Y)=0\, \forall Y\in T_xM\}.
$$

The dimension of \,$\Delta(x)$\, is called the \emph{index of relative nullity} of \,$f$\, at \,$x,$\, and is
denoted by \,$\nu(x).$\, We also define the \emph{index of minimum relative nullity} of \,$f$\, by
$$
\nu_{\min}=\min_{x\in M}\nu(x).
$$

Finally, let us recall that an isometric immersion \,$f:M^n\rightarrow Q_c^{n+p}$\, is called \emph{rigid} if, for any isometric
immersion \,$g:M^n\rightarrow Q_c^{n+p},$\, there is an ambient isometry
\,$\Phi:Q^{n+p}\rightarrow Q^{n+p}$\, such that \,$g=\Phi\circ f.$\,

Most of our results on rigidity here will follow from the following theorem, due to Sacksteder \cite{sacksteder2}.

\begin{sacks}
Let \,$f:M^n\rightarrow Q_c^{n+1}$\, be an isometric immersion of a compact (resp. complete) Riemannian manifold with
\,$n\ge 3$\, and \,$c\le 0$ (resp. \,$n\ge 4$\, and \,$c>0$). Then,  \,$f$\, is rigid, provided its
set of totally geodesic points does not disconnect \,$M.$
\end{sacks}

\section{Semi-Regular Isometric Immersions} \label{sec-semiregular}

We say that an isometric immersion \,$f:M^n\rightarrow{Q}_c^{n+p}$\, admits a \emph{reduction of codimension} to \,$q<p,$\, if there is
a totally geodesic submanifold \,$Q_c^{n+q}\subset Q_c^{n+p}$\, such that \,$f(M)\subset Q_c^{n+q}.$\, The immersion  \,$f$\, is also said to be
\,$(1\,;q)$-\emph{regular} if, for
all \,$x\in M,$\, the \emph{first normal space} of \,$f$\, at \,$x,$
$$
\mathscr{N}(x):={\rm span}\{\alpha_f(X,Y) \,;\, X,Y\in T_xM\},
$$
has constant dimension \,$q.$\, In this case,
$$\mathscr{N}:=\bigcup_{x\in M}\mathscr{N}(x)$$
is a subbundle of  \,$TM^\perp$\, which will be called
the \emph{first normal bundle} of \,$f.$\, $\mathscr{N}$\, is said to
be \emph{parallel} if, at any point \,$x\in M,$\, one has
$$\nabla_X^\perp\xi\in\mathscr{N}(x) \,\,\,\, \forall\, X\in T_xM, \,\, \xi\in\mathscr{N}(x),$$
where \,$\nabla^\perp$\, stands for the normal connection of \,$f.$

For future reference, let us quote a  standard result on reduction of codimension of isometric immersions.

\begin{proposition} \label{prop-reductioncodimension}
Let  \,$f:M^n\rightarrow{Q}_c^{n+p}$\, be a connected
\,$(1\,;q)$-regular isometric immersion whose first normal bundle \,$\mathscr{N}$\, is parallel.
Then, \,$f$\, admits a reduction of codimension to \,$q.$
\end{proposition}

\begin{proof}
See Corollary 4.2 of \cite{dajczer}.
\end{proof}

For our purposes, it will be convenient to introduce the following  concept.

\begin{definition}
An isometric immersion \,$f:M^n\rightarrow{Q}_c^{n+p}$\, will be called  $(1\,;q)$-\emph{semi-regular}, if
it is non-totally geodesic and its restriction to \,${M-M_{\rm tot}}$\, is $(1\,;q)$-regular.
\end{definition}

We establish now two elementary results regarding semi-regular isometric immersions.
The first one appears in \cite{rodriguez-tribuzy} (Lemma 2). For the reader's convenience,
we will present it here with a slightly different proof.

First, recall the \emph{Codazzi equation}  for isometric immersions \,$f:M^n\rightarrow Q_c^{n+p}:$\,
$$
\left(\nabla_{X}^\perp\alpha_f\right)(Y,Z)=\left(\nabla_{Y}^\perp\alpha_f\right)(X,Z), \,\,\, X,Y, Z\in TM,
$$
where
$$
\left(\nabla_{X}^\perp\alpha_f\right)(Y,Z):=\nabla_{X}^\perp\alpha_f(Y,Z)-\alpha_f\left( \nabla_{X}Y,Z \right)-\alpha_f\left(Y, \nabla_{X}Z \right).
$$

\begin{proposition} \label{prop-parallel}
Let  \,$f:M^n\rightarrow Q_c^{n+p}$\, be an $(1\,;1)$-semi-regular isometric immersion such that  \,$M-M_c$\, is nonempty.
Then, \,$\mathscr{N}$\, is parallel in \,$M-M_c$\,.
\end{proposition}
\begin{proof}
Given \,$x\in M-M_c$\,, let \,$\xi$\, be a unit normal field which spans \,$\mathscr{N}$\, in an open neighborhood \,$V$\, of
\,$x$\, in \,$M-M_c$\,, and  \,$\{X_1\,, \dots ,X_n\}$\,  an orthonormal frame in \,$TV$\, that diagonalizes
\,$A_\xi$\, with corresponding eigenvalues \,$\lambda_1\,,\dots ,\lambda_n$\,. In this setting,  one has
\begin{equation} \label{eq-sff}
\alpha_f(X_i,X_j)=\langle\alpha_f(X_i,X_j),\xi\rangle\xi=\langle A_\xi X_i,X_j\rangle\xi=\delta_{ij}\lambda_i\xi.
\end{equation}

By Gauss equation, for all \,$i\ne j\in\{1,\dots ,n\},$\,
\begin{equation} \label{eq-gauss}
K_{\scriptscriptstyle{M}}(X_i\,,X_j)=c+\lambda_i\lambda_j\,.
\end{equation}
Hence, at each point of  \,$V\subset M-M_c$\,, \,$A_\xi$\, has at least two nonzero eigenvalues, that is, for a given
\,$i\in\{1,\dots  ,n\},$\,  there exists \,$j\ne i,$\, such that
\,$\lambda_j\ne 0.$\, Moreover, from \eqref{eq-sff}, we have that  \,$\alpha_f(X_i,X_j)=0$\, and \,$\alpha_f(X_j,X_j)=\lambda_j\xi.$\,
Thus, the following equalities hold:
\begin{itemize}[parsep=1ex]
\item $\left(\nabla_{X_j}^\perp\alpha_f\right)(X_i,X_j)=-\alpha_f\left( \nabla_{X_j}X_i,X_j \right)-\alpha_f\left(X_i\,, \nabla_{X_j}X_j \right).$

\item $\left(\nabla_{X_i}^\perp\alpha_f\right)(X_j,X_j)=(X_i\lambda_j)\xi+\lambda_j\nabla_{X_i}^\perp\xi-2\alpha_f\left( \nabla_{X_i}X_j,X_j \right).$
\end{itemize}
\vt

Now, the Codazzi equation gives that the right hand sides of these two equalities coincide,
which  implies that
\,$\nabla_{X_i}^\perp\xi(x)\in\mathscr{N}(x).$\, Therefore, the first normal bundle
\,$\mathscr{N}$\, is parallel in \,$M-M_c$\,.
\end{proof}

\begin{proposition} \label{prop-sdsr}
Let \,$f:M^n\rightarrow Q_c^{n+p}$\, be a non-totally geodesic  isometric immersion with
semi-definite second fundamental form,  and denote by \,$H$\, its mean curvature field. Then, the following hold:
\begin{itemize}[parsep=1ex]
\item[{\rm i)}]  \,$x\in M_{\rm tot}$\, if and only if \,$H(x)=0.$

\item[{\rm ii)}]  For \,$x\in M-M_{\rm tot}$\,, \,$\mathscr{N}(x)={\rm span}\,\{H(x)\}$\,
and all the eigenvalues of \,$A_H$\, are nonnegative. In particular, \,$f$\, is  $(1\,;1)$-semi-regular.

\item[{\rm iii)}] \,$K_{\scriptscriptstyle{M}}\ge c$\, everywhere.
\end{itemize}
\end{proposition}

\begin{proof}
By definition, one has
\begin{equation}\label{eq-meancurvature}
H=\sum_{i=1}^{n}\alpha_f(X_i\,,X_i)=\sum_{j=1}^p(\text{trace}\, A_{\xi_j})\xi_j\,,
\end{equation}
where \,$\{X_1\,,\dots ,X_n\}$\, and \,$\{\xi_1\,, \dots ,\xi_p\}$\, are arbitrary orthonormal frames
in \,$TM$\, and \,$TM^\perp,$\, respectively.

Since \,$f$\, has semi-definite second fundamental form,
a  shape operator \,$A_{\xi}$\, of \,$f$\,  is identically zero if and only if its trace is equal to zero.
Consequently, \,$f$\, is totally geodesic at \,$x\in M$\,
if  and only if  \,$H(x)=0,$\, which proves (i).

For  \,$x\in M-M_{\rm tot}$\,,  we can choose
an orthonormal frame \,$\{\xi_1\,, \dots ,\xi_p\}\subset TM^\perp$\, in an open neighborhood
\,$U$\, of \,$x,$\,
with  \,$\xi_1={H}/{\|H\|}.$\, In this case, \eqref{eq-meancurvature} yields
$$\text{trace}\, A_{\xi_1}=\|H\|>0 \quad\text{and}\quad \text{trace}\,A_{\xi_j}=0 \,\, (\text{i.e.} \,\,  A_{\xi_j}=0) \,\, \forall j=2, \dots ,p.$$
Therefore, for \,$X, Y\in TU,$\, one has
$$
\alpha_f(X,Y)=\sum_{j=1}^{p}\langle\alpha_f(X,Y),\xi_j\rangle\xi_j=\sum_{j=1}^{p}\langle A_{\xi_j}X,Y\rangle\xi_j=\langle A_{\xi_1}X,Y\rangle\xi_1\,,
$$
which implies that  \,$\mathscr{N}(x)={\rm span}\,\{H(x)\}.$\, Since the trace of \,$A_{\xi_1}$\, is positive, we also have that
all of its eigenvalues are nonnegative. This proves (ii).

Now, considering  an orthonormal basis of
\,$T_xM$\, that diagonalizes \,$A_{H}$\, and the Gauss equation (as in \eqref{eq-gauss}), we easily
conclude that \,$K_{\scriptscriptstyle{M}}\ge c$\, at \,$x.$\,
Since \,$K_{\scriptscriptstyle{M}}=c$\, on \,$M_{{\rm tot}}$\,, we have that
\,$K_{\scriptscriptstyle{M}}\ge c$\, on all of \,$M,$\, which proves (iii).
\end{proof}

\begin{remark}  \label{rem-nullity}
By means of the Gauss equation, as in  the last paragraph of the above proof, we conclude that
for \,$x\in M_c-M_{\rm tot}$\,,  there is exactly one nonzero eigenvalue of \,$A_H$\,.
So, \emph{for any  $(1\,;1)$-{semi-regular} isometric immersion}
\,$f:M^n\rightarrow{Q}_c^{n+p},$\,
$$\nu(x)=n-1 \,\,\, \forall x\in M_c-M_{\rm tot}\,.$$
\end{remark}

\section{Beltrami Maps}  \label{sec-beltramimaps}

Given a point \,$e$\, in \,$S^n\subset\R^{n+1},$\, let \,$\mathcal{H}_{e}$\, be the open hemisphere of \,$S^n$\, centered at \,$e,$\, that is,
the open geodesic ball of \,$S^n$\, with center at \,$e$\, and radius \,$\pi/2.$\,
The central projection \,$\varphi$\, from  \,$\mathcal{H}_{e}$\, to
the tangent space of \,$S^n$\, at \,$e,$\, which we identify with \,$\R^n,$\,  is a diffeomorphism called the \emph{Beltrami map} of  \,$\mathcal{H}_e$\,.

For   \,$e=(0,\dots ,0,1),$\,
this map  is defined as
\begin{equation} \label{eq-beltramimap}
\begin{array}{cccc}
\varphi: & \mathcal{H}_e & \rightarrow & \R^n\\
         &     x         & \mapsto     & \frac{x}{x_{n+1}}
\end{array},
\end{equation}
where \,$x_{n+1}$\, stands for the $(n+1)$-th coordinate of \,$x$\, in \,$\R^{n+1}.$\,

It is easily seen that \,$\varphi$\, and its inverse are both  geodesic maps,
that is, they take geodesics to geodesics and, in particular,
convex sets to convex sets.

Now, consider the  Lorentzian space \,$\mathbb{L}^{n+1}=(\R^{n+1}, \langle \,,\, \rangle_{\scriptscriptstyle{\mathbb{L}}}),$\,
where
$$
 \langle x\,, x \rangle_{\scriptscriptstyle{\mathbb{L}}}=\sum_{i=1}^{n}x_i^2 -x_{n+1}^2\,, \,\, x=(x_1\,,\dots ,x_{n+1}),
$$
and the hyperboloid model of hyperbolic space
$$
\h^n=\{x\in\mathbb{L}^{n+1};  \langle x\,,x \rangle_{\scriptscriptstyle{\mathbb{L}}}=-1 \,\,\,\text{and} \,\,\, x_{n+1}>0 \}.
$$

Let us   identify the  affine subspace \,$x_{n+1}=1$\, of \,$\mathbb{L}^{n+1}$\, with \,$\R^n,$\,
and denote its unit open ball centered at \,$0$\, by \,$B^n.$\,
Considering in $B^n$\, the Euclidean  metric,
one has that the  \emph{Beltrami map} of \,$\h^n,$\,
\begin{equation} \label{eq-beltramimap10}
\begin{array}{cccc}
\varphi: & \mathbb{H}^n & \rightarrow & B^n\\
         &     x         & \mapsto     & \frac{x}{x_{n+1}}
\end{array},
\end{equation}
is also a geodesic diffeomorphism.
This follows from the well known fact that, endowed with a suitable metric,
\,$B^n$\, represents the Kleinian model of hyperbolic space. In this setting,  the map \,$\varphi$\, becomes  an
isometry (in particular, a geodesic diffeomorphism)  between the Lorentzian and the Kleinian models of \,$\h^n.$\,
However, the geodesics of the Kleinian model  are precisely the segments of Euclidean
lines which lie in \,$B^n.$

The following result, which generalizes Proposition 2.3 of \cite{docarmo-warner}, will play a fundamental
role in the proofs of our  theorems.

\begin{proposition} \label{prop-beltramimaps}
Let  \,$\mathcal{H}^{n+p}$\,  be either an open hemisphere of \,$S^{n+p}$\, or the hyperbolic space \,$\h^{n+p},$\, and
\,$\varphi:\mathcal{H}^{n+p}\rightarrow\R^{n+p}$\, its corresponding Beltrami map. Suppose that
\,$f:M^n\rightarrow\mathcal{H}^{n+p}$\, is an isometric immersion with second fundamental form \,$\alpha_f$\,.
Consider also in \,$M$\, the metric induced by the immersion
$$\bar{f}:=\varphi\circ f:M^n\rightarrow\R^{n+p}$$
and denote its second fundamental form by \,$\alpha_{\bar f}$\,.
Under these conditions, the following hold:

\begin{itemize}[parsep=1ex]

\item[{\rm i)}] At a point \,$x\in M,$\, \,$\alpha_f$\, is semi-definite (resp. definite) if and only
if \,$\alpha_{\bar f}$\, is semi-definite (resp. definite).

\item[{\rm ii)}] The set of totally geodesic points of \,$f$\, and \,$\bar{f}$\, coincide.

\end{itemize}

\end{proposition}

\begin{proof}
In the spherical case,  we may assume without loss of generality that
\,$\mathcal{H}^{n+p}$\, is the open hemisphere of \,$S^{n+p}$\, centered at
\,$e=(0,\dots ,0, 1)\in\R^{n+p+1}.$\, In this way, we can write
\begin{equation} \label{eq-beltramidef}
\varphi(x)=\phi(x)x, \quad\phi(x)=\frac{1}{x_{n+p+1}}\,, \,\,\,\,  x=(x_1\,, \dots , x_{n+p+1})\in \mathcal{H}^{n+p},
\end{equation}
and treat both the spherical and hyperbolic cases simultaneously.

Recall that the Riemannian connection of \,$\R^{n+p+1},$\,
which we will denote by \,$\widehat{\nabla},$\,
is the same for both the Euclidean and Lorentzian metrics. So, denoting
the  Riemannian connection of \,$\mathcal{H}^{n+p}$\,  by
\,$\widetilde{\nabla},$\,  one has
\begin{equation} \label{eq-connections}
\widehat{\nabla}_XY=\widetilde{\nabla}_XY-\langle X,Y\rangle I \,\,\,\, \forall X, Y\in T\mathcal{H}^{n+p},
\end{equation}
where \,$I$\, stands for the identity field of \,$\R^{n+p+1}$\, and \,$\langle\,,\rangle$\, is either
the Euclidean or Lorentzian metric of \,$\R^{n+p+1}.$

Denote by \,$\overline{M}$\, the manifold \,$M$\, endowed with the metric induced by \,$\bar{f}.$\,
Considering \eqref{eq-beltramidef}, we have for all \,$X\in TM$\, that
\begin{equation} \label{eq-tangent}
\overline{X}:=\varphi_* X=(X\phi)I+\phi X\in T\overline{M}\subset\Pi,
\end{equation}
where \,$\Pi$\, denotes the Euclidean orthogonal complement of \,$e$\, in \,$\R^{n+p+1}.$

Now, given \,$\xi\in TM^\perp\subset T\mathcal{H}^{n+p},$\, one has   \,$\langle \xi, I\rangle= 0.$\,
Thus, writing \,$\overline{\xi}$\, for  the component
of \,$\xi$\, in \,$\Pi,$\, that is,
\begin{equation}  \label{eq-orthogonal}
\overline{\xi}:=\xi-(\xi_{n+p+1}) e, \,\,\,\, \xi\in TM^\perp,
\end{equation}
we have that \,$\overline{\xi}\in T\overline{M}^\perp\subset\Pi.$\, Indeed, for all
\,$\overline{X}=\varphi_* X\in T\overline{M},$\,  one has
$$
\langle\overline{\xi},\overline{X}\rangle = \langle \xi,\overline{X}\rangle = \langle \xi,(X\phi)I+\phi X\rangle =0.
$$

Let us  fix a point \,$x\in M.$\,
Since \,$\varphi$\, is a diffeomorphism and no tangent space to
\,$\mathcal{H}^{n+p}$\, is orthogonal to \,$\Pi,$\, for a suitable open neighborhood \,$U$\, of \,$x$\,,
the correspondences defined in \eqref{eq-tangent} and \eqref{eq-orthogonal},
$$
X\in TU\mapsto\overline{X}\in T\overline{U} \quad\text{and}\quad \xi\in TU^\perp\mapsto\overline{\xi}\in T\overline{U}^\perp,
$$
are bundle isomorphisms.

Now, extending a given \,$\xi\in TU^\perp$\,  by parallel displacement along the lines through the
origin of \,$\R^{n+p+1},$\, we have that \,$\widehat{\nabla}_I\xi=0.$\,
Thus,
\begin{equation}\label{eq-beltrami1}
\widehat{\nabla}_{\overline{X}}\,\xi=\phi\widehat{\nabla}_X\xi\, \quad  \forall\,\overline{X}=\varphi_* X\in T\overline{U}.
\end{equation}
Furthermore, given \,$X\in TU,$\, one  has
\begin{equation}\label{eq-beltrami2}
0=X\langle \xi, I\rangle=\langle\widehat{\nabla}_X \xi, I\rangle+\langle\xi,\widehat{\nabla}_XI\rangle=
\langle\widehat{\nabla}_X \xi, I\rangle+\langle\xi,X\rangle
=\langle\widehat{\nabla}_X \xi, I\rangle
\end{equation}
and, from \eqref{eq-connections}, that
\,$\widehat{\nabla}_{{X}}\,{\xi}=\widetilde{\nabla}_{{X}}\,{\xi}.$\, This, together with
\eqref{eq-beltrami1} and \eqref{eq-beltrami2}, yields
$$
\langle\alpha_{\bar{f}}(\overline{X}\,,\overline{X}),\overline{\xi}\rangle =
-\langle\widehat{\nabla}_{\overline{X}}\,\overline{\xi}, \overline{X}\rangle=
-\langle\widehat{\nabla}_{\overline{X}}\,{\xi}, \overline{X}\rangle=-\phi^2\langle\widehat{\nabla}_{X} \xi,  X\rangle
=-\phi^2\langle\widetilde{\nabla}_{X} \xi,  X\rangle,
$$
which implies that the equality
\begin{equation} \label{eq-convexity1}
\langle\alpha_{\bar{f}}(\overline{X},\overline{X}),\overline{\xi}\rangle=\phi^2\langle\alpha_f({X},{X}),\xi\rangle
\end{equation}
holds everywhere in \,$U$\, and, in particular, at \,$x.$
Since the function \,$\phi$\, never vanishes, the assertions (i) and (ii)
easily follow.
\end{proof}

\begin{remark} \label{rem-curvature}
In the above proof, equation \eqref{eq-convexity1} gives that the ranks of the shape operators
\,$A_\xi$\, and \,$A_{\bar\xi}$\, coincide.
Consequently, when the codimension  \,$p$\, is \,$1$\, and \,$M$\, is oriented,
\,$x\in M$\, is a point of vanishing Gauss-Kronecker curvature for \,$f$\, if and only if it is for \,$\bar{f}=\varphi\circ f$\,
(recall that the \emph{Gauss-Kronecker curvature} of an oriented hypersurface \,$f:M^n\rightarrow{Q}_c^{n+1}$\,
is the determinant of  its shape operator).
\end{remark}

By combining Proposition \ref{prop-beltramimaps} with an outstanding  result by Guan and Spruck \cite{guan-spruck}, we obtain the following
existing result for hypersurfaces with prescribed boundary and vanishing Gauss-Kronecker curvature in space forms.

\begin{corollary} \label{cor-prescribedboundary}
Let \,$\Gamma^{n-1}\subset\mathcal{H}^{n+1}$\, be  a compact
embedded (not necessarily connected) $(n-1)$-dimensional submanifold of \,$\mathcal{H}^{n+1}.$\, Assume that
there exists a connected compact oriented \,$C^2$\, hypersurface with boundary \,$g:N^n\rightarrow\mathcal{H}^{n+1},$\,
satisfying \,$g(\partial N)=\Gamma,$\,
whose second fundamental form  is semi-definite on \,$N$\, and definite  in a neighborhood of \,$\partial N.$\,
Under these conditions, there exists a connected compact oriented  \,$C^{1,1}$\, (up to the boundary) hypersurface
\,$f:M^n\rightarrow\mathcal{H}^{n+1}$\, with semi-definite  second fundamental form,
which satisfies \,$f(\partial M)=\Gamma$\,
and has vanishing Gauss-Kronecker curvature everywhere.
\end{corollary}

\begin{proof}
It follows from Proposition \ref{prop-beltramimaps}
that  the second fundamental form of the immersion \,$\bar{g}=\varphi\circ g:N\rightarrow\R^{n+1}$\,  is semi-definite,
and definite in a neighborhood of \,$\partial N.$\, Thus,
\,$\overline{\Gamma}:=\varphi({\Gamma})=\bar{g}(\partial N)$\, is a compact embedded submanifold of
\,$\R^{n+1}$\, which fulfills the hypothesis of Theorem 1.2 in \cite{guan-spruck}. Therefore, there exists a connected compact
oriented    \,$C^{1,1}$\, (up to the boundary) hypersurface with boundary,
\,$\bar{f}: M\rightarrow\R^{n+1},$\, whose second fundamental form is semi-definite,
which  satisfies \,$\bar{f}(\partial M)=\overline{\Gamma}$\,
and has vanishing Gauss-Kronecker curvature everywhere. Hence (see Remark \ref{rem-curvature}),
$$f=\varphi^{-1}\circ\bar{f}:M^n\rightarrow\mathcal{H}^{n+1}$$
is the desired hypersurface.
\end{proof}

An isometric immersion \,$f:M^n\rightarrow\tilde{M}^{n+p}$\,
is said to have the \emph{convex hull property}
if, for every domain \,$D$\,  on \,$M$\, such that \,$f(D)$\, is bounded in \,$\tilde{M},$\,
\,$f(D)$\,  lies in the convex hull of its boundary in \,$\tilde{M}.$\,
It is a well known fact that minimal immersions in Euclidean space have this property.

In \cite{osserman}, R. Osserman established that an isometric immersion into Euclidean space
has the convex hull property if and only if there is no normal direction in which
its second fundamental form is definite. Since Beltrami maps are convexity-preserving, Osserman's theorem and
Proposition \ref{prop-beltramimaps} give the following result, which, in the spherical case, generalizes a theorem due
to B. Lawson (Theorem 1' in \cite{lawson}).

\begin{corollary} \label{cor-interiorCHP}
An isometric immersion \,$f:M^n\rightarrow\mathcal{H}^{n+p}$\, has the convex hull property if and only if
there is no normal direction \,$\xi$\, for which the $2$-form
$$
(X,Y)\in TM\times TM\mapsto\langle\alpha_f(X,Y),\xi\rangle
$$
is definite. In particular, any minimal isometric immersion \,$f:M^n\rightarrow\mathcal{H}^{n+p}$\, has the
convex hull property.
\end{corollary}

In \cite{alexander-ghomi}, Alexander and Ghomi showed that, in Euclidean space,
the convex hull of the boundary of a hypersurface  whose second fundamental form
is semi-definite has a property which is a dual of the classical one
we considered above. This result, together with Proposition \ref{prop-beltramimaps}, yields
the following

\begin{corollary} \label{cor-exteriorCHP}
Let \,$f:M^n\rightarrow\mathcal{H}^{n+1}$\,  be a compact connected hypersurface with boundary,
whose second fundamental form \,$\alpha_f$\, is semi-definite.
Let \,$C$\, be the convex hull of \,$f(\partial M)$\, and assume that:
\begin{itemize}[parsep=1ex]
\item[{\rm i)}]  $f(\partial M)\subset\partial C.$
\item[{\rm ii)}]   $\alpha_f$\, is definite in a neighborhood of \,$\partial M.$
\item[{\rm iii)}]  \,$f$\,  is an embedding on each component of \,$\partial M.$\,
\end{itemize}
Then, the image of the interior of \,$M$\, lies completely outside \,$C,$\, that is,
$$
f({\rm int}\,M)\cap C=\emptyset.
$$
\end{corollary}

We remark that, in Corollary \ref{cor-exteriorCHP},
none of the conditions (i)---(iii) can be omitted (see  \cite{alexander-ghomi}).

\section{Proofs of Theorems \ref{th-convexity} --- \ref{th-hypersurface}} \label{sec-proof1and2}

We proceed now to the proofs of the theorems.
To show  Theorem \ref{th-convexity}, we    recover the Euclidean case by
means of  Proposition \ref{prop-beltramimaps}, and then apply   Jonker's Theorem \cite{jonker}, stated below.
The same goes to Theorem \ref{th-hypersurface}, in which Hadamard Theorem \cite{hadamard1}  plays  the role of Jonker's Theorem.
Theorem \ref{th-hconvexity} will be derived from  propositions \ref{prop-reductioncodimension} and \ref{prop-parallel},
Theorem \ref{th-convexity}, and the main results of \cite{currier}.

\begin{jt}
Let \,$M^n$\, be a complete connected Riemannian manifold, and  \,$f:M^n\rightarrow\R^{n+p}$\,   an
isometric immersion with semi-definite second fundamental form. Then, one of the
following two possibilities holds:

\begin{itemize}[parsep=1ex]

\item[{\rm i)}] There is an affine subspace \,$\R^{n+1}$  of  \,$\R^{n+p}$\,  such that \,$f$\, embeds
\,$M$\, onto the boundary of a convex set of \,$\R^{n+1}.$\, In particular, \,$M$\, is diffeomorphic to
\,$S^n$ or to \,$\R^n.$

\item[{\rm ii)}] The immersion \,$f$\, is an $(n-1)$-cylinder over a curve \,$\gamma:\R\rightarrow\R^{p+1},$\, that is,
\,$M$\, and \,$f$\, split respectively as \,$\R\times\R^{n-1}$ and  \,$\gamma\times {\rm id},$\,
where \,${\rm id}$\, stands for the identity map of \,$\R^{n-1}.$

\end{itemize}
\end{jt}

\begin{remark}\label{rem-jonker}
Along the proof of Jonker's Theorem, it is shown that  the possibility (i) occurs only if \,$M-M_{{\rm tot}}$\, is connected.
\end{remark}

\begin{proof}[Proof of Theorem \ref{th-convexity}]
Let us prove first that, for   \,$c=1,$\,
there is an open hemisphere of \,$S^{n+p}$\, which contains \,$f(M).$

Since \,$\alpha_f$\, is semi-definite,
by Proposition \ref{prop-sdsr},
the first normal space of \,$f$\, at any point \,$x\in M-M_{\rm tot}$\,
is spanned by the mean curvature vector \,$H(x).$\, In this case,
there is a point \,$x\in M-M_{\rm tot}$\,
at which  \,$\alpha_f$\, is positive-definite in the direction \,$H(x),$\,  that is,  all the
eigenvalues of \,$A_H$\, at \,$x$\, are  positive. Indeed, if it were not so,
the index of minimum relative nullity of \,$f,$\, \,$\nu_{\min}$\,,
would satisfy \,$0<\nu_{\min}<n.$\,
Then, by a result due to  Dajczer and Gromoll \cite{dajczer-gromoll}, at any \,$x\in M$\, such that \,$\nu(x)=\nu_{\min},$\,
the number of positive and negative eigenvalues of \,$A_H$\,  would be equal. However, by Proposition \ref{prop-sdsr},
all the eigenvalues of \,$A_H$\, are nonnegative.

Let then \,$x\in M-M_{\rm tot}$\, be a point  at which  \,$\alpha_f$\, is positive-definite in the direction \,$H(x).$\,
Denote by  \,$\mathcal{H}_{f(x)}$\, the open hemisphere
of \,$S^{n+p}$\, centered at \,$f(x),$\, and let
\,$\varphi:\mathcal{H}_{f(x)}\rightarrow\R^{n+p}$\, be its  Beltrami map.
Write \,$N$\, for  the connected component of
\,$f^{-1}(\mathcal{H}_{f(x)})$\, that contains \,$x,$\,  and endow it with the metric induced by the immersion
\,$
\bar f=\varphi\circ f|_N:N\rightarrow\R^{n+p}.
$\,
By Proposition \ref{prop-beltramimaps},
\,$\bar{f}$\, is an isometric immersion with  semi-definite second fundamental form. Furthermore,
as shown by do Carmo and Warner
in Lemma 2.5 of \cite{docarmo-warner}, which is valid regardless of the codimension,  \,$\bar f$\, is
also  complete.

It follows from these  considerations and  Jonker's Theorem that  \,$\bar{f}(N)\subset\R^{n+p}$\, is either a cylinder over a curve
or the boundary of a convex set in an $(n+1)$-dimensional  affine subspace of
\,$\R^{n+p},$\, which we identify with \,$\R^{n+1}$\,. In the latter case, \,$\bar{f}|_N:N\rightarrow\R^{n+1}$\, is an embedding.
However, by Proposition \ref{prop-beltramimaps}, \,$\alpha_{\bar{f}}$\, is definite at \,$x,$\, which excludes the
possibility of \,$\bar{f}(N)$\, being a cylinder.

The definiteness of  \,$\alpha_{\bar{f}}$\,  at \,$x$\,  also implies that
there exists an open neighborhood \,$U$\, of \,$x$\, in \,$M,$\, such that
\,$\bar{f}({\rm cl}(U)-\{x\})$\, is contained in  one  of the open semi-spaces
of \,$\R^{n+1}$\, determined by  \,$\Sigma:=\bar{f}_*(T_xM),$\, say
\,$\Sigma_+$\,.


Let \,$\eta\in\R^{n+1}$\, be the unit normal to \,$\bar{f}(M)$\, at \,$\bar{f}(x)$\, pointing to \,$\Sigma_+$\,. If we
write \,$\Pi$\, for the orthogonal complement of \,$\eta$\, in \,$\R^{n+p},$\, we have that
\,$\bar{f}({\rm cl}(U)-\{x\})$\, is contained in the open semi-space \,$\Pi_+$\, of \,$\R^{n+p}$\, which is determined by \,$\Pi$\,
and  contains \,$\Sigma_+$\,.
Hence, denoting  by  \,$\mathcal{H}$\,  the open  hemisphere
of \,$S^{n+p}$\, that contains \,$\varphi^{-1}(\Pi_+),$\,  we have that
\,$f(x)\in\partial\mathcal{H}$\, and
\,$f({\rm cl}\,U-\{x\})\subset \mathcal{H}.$\,

Now, notice that \,$\partial\mathcal{H}$\, is a local
$(n+p-1)$-dimensional ``supporting sphere'' for \,$f(M)$\, at \,$f(x).$\,
This fact will allow us to adapt the reasoning of the last two paragraphs
of the proof of Theorem 1.1(a) in \cite{docarmo-warner} to conclude
that, in fact,  \,$f(M)-\{f(x)\}\subset \mathcal{H}.$

The argument goes as follows.
Let  \,$C\subset f^{-1}(\mathcal{H})\subset M$\, be the connected component of
\,$f^{-1}(\mathcal{H})$\, that contains \,${\rm cl}\,U-\{x\}.$\, By abuse of notation, write
\,$\varphi$\, for the Beltrami map of \,$\mathcal{H}$\, and
set \,$\bar{f}=\varphi\circ f:C\rightarrow\R^{n+p}.$\, Thus, \,$\bar{f}$\, is a noncompact complete and  non-flat
isometric immersion with semi-definite second fundamental form.

By Jonker's Theorem, \,$\bar{f}(C)$\, is diffeomorphic
to \,$\R^n.$\, Assume that
\,$\partial U$\, is a geodesic sphere of \,$M$\,
centered at \,$x$\, so that
\,$\bar{f}(\partial U)$\, separates \,$\bar{f}(C)$\, into two connected components, being one of them
bounded. Denoting by \,$\Omega$\, this bounded component, we have that \,$\bar{f}(C-{\rm cl}\,U)\subset\Omega.$\, Otherwise,
there would be \,$y\in C-{\rm cl}\,U$\, such that \,$\bar{f}(y)\in \bar{f}(C)-{\rm cl}\,\Omega.$\, Then, a minimal geodesic
\,$\gamma$\, joining \,$y$\, to \,$x$\, would  cross \,$\partial U$\,  only once, and its image  \,$\bar{f}(\gamma)$\, would be unbounded.
This, however,  is a contradiction, for  \,$\bar{f}(\gamma)$\,  comes from the unbounded component
\,$f(C)-{\rm cl}\,\Omega,$\, and  cross \,$\partial\Omega$\, only once. Therefore,
\,$f(x)$\, is the only limit point of \,$f(C)$\, on \,$\partial\mathcal{H},$\, that is,
\,$f(C)\subset f({\rm cl}\,U-\{x\})\subset\mathcal{H}.$\, Thus, since
\,$M$\, is connected, one has \,$f(M)-\{f(x)\}\subset \mathcal{H},$\, as claimed.

Finally, choosing \,$e\in S^{n+p}$\, sufficiently close to the center of \,$\mathcal{H},$\,
we have that \,$f(M)$\, is contained in the open hemisphere of \,$S^{n+p}$\, centered at \,$e,$
as we wished to prove.

Suppose now that \,$c\ne 0$\, and let \,$\mathcal{H}^{n+p}$\,
be either an open hemisphere of \,$S^{n+p}$\, that contains \,$f(M)$\, or the hyperbolic space \,$\h^{n+p}.$\,
As before, if we write
$$\varphi:\mathcal{H}^{n+p}\rightarrow\R^{n+p}$$
for the Beltrami map of \,$\mathcal{H}^{n+p},$\, we have that
the second fundamental form of the immersion \,$\bar{f}=\varphi\circ f$\, is semi-definite.
Thus, since \,$M$\, is compact,
it follows from Jonker's Theorem  that \,$\bar{f}$\, is an embedding and \,$\bar{f}(M)$\,
is the boundary of a compact convex set in an $(n+1)$-dimensional
affine subspace \,$\R^{n+1}\subset\R^{n+p}.$\, Therefore,
\,$f=\varphi^{-1}\circ\bar{f}$\, embeds \,$M$\, onto
the boundary of a compact convex  set of \,$\mathcal{H}^{n+1}=\varphi^{-1}(\R^{n+1})\subset\mathcal{H}^{n+p},$\,
for \,$\varphi$\, is  convexity-preserving.

Since we have reduced the codimension to one, we can apply
do Carmo and Warner Theorem \cite{docarmo-warner} to conclude that \,$f$\, is rigid for \,$c=1.$
For \,$c=-1$\, and \,$n=2,$\, the rigidity of \,$f$\, follows from
Theorem 5 of \cite{fomenko-gajubov}, which, in fact, is set in the more general
context of compact surfaces with  boundary.

Suppose then that \,$c=-1$\, and \,$n > 2.$\, Since the set of totally geodesic points of \,$\bar{f}$\, does not disconnect \,$M$\,
(see Remark \ref{rem-jonker}),  it follows from Proposition \ref{prop-beltramimaps}-(ii) that the same is true for \,$f.$\,
Thus, Sacksteder Rigidity Theorem applies and gives that \,$f$\, is rigid. This concludes our proof.
\end{proof}

\begin{proof}[Proof of Theorem \ref{th-hconvexity}]
It follows from  theorems A and B of \cite{currier} that, except for the rigidity in the compact case,
the result is true if the codimension  of \,$f$\,  is equal to one.
Otherwise,  we can reduce the codimension to one.

Indeed, since \,$\alpha_f$\, is clearly  positive-definite, we have that
\,$M_c$\, is empty  and, by Proposition \ref{prop-sdsr}, that
\,$f$\, is $(1\,;1)$-regular. So, by Proposition \ref{prop-parallel}, the first normal bundle of \,$f$\, is parallel. Since
\,$M$\, is connected, Proposition \ref{prop-reductioncodimension} gives that \,$f$\,  admits a reduction of codimension to one.

The rigidity of \,$f$\, when \,$M$\, is compact follows from Theorem \ref{th-convexity}.
\end{proof}

\begin{proof}[Proof of Theorem \ref{th-hypersurface}]
Considering  Proposition \ref{prop-beltramimaps} together with its notation,
we have that \,$\alpha_f$\, is definite  if and only if \,$\alpha_{\bar{f}}$\, is definite.
But, by  Hadamard Theorem \cite{hadamard1} (see also \cite{dajczer} -- Chapter 2), \,$\alpha_{\bar{f}}$\, is definite
if and only if \,$\bar{f},$\, and so \,$f,$\, has non-vanishing Gauss-Kronecker curvature (see Remark \ref{rem-curvature}).
This proves the equivalence between (i) and (ii).

Now, notice that  \,$\psi={\bar{\xi}}/{\|\bar{\xi}\|}$\, is nothing but the Gauss map of \,$\bar{f}.$\,
So, again by Hadamard Theorem, \,$M$\, is orientable and
\,$\psi$\, is  a diffeomorphism if and only if \,$\alpha_{\bar{f}}$\,, and so \,$\alpha_f$\,, is
definite, which shows that (i) is also equivalent to (iii).

The last assertion follows directly from Theorem \ref{th-convexity}.
\end{proof}

\section{Proof of Theorem \ref{th-semiregular}} \label{sec-proof3}

The proof of Theorem \ref{th-semiregular} will be based on some properties of
the relative nullity distribution that we shall  introduce in the following.

Let \,$f:M^n\rightarrow \tilde Q_c^{n+p}$\, be an isometric immersion
whose index of relative nullity
is constant and equal to \,$\nu_0>0$\, in some open set \,$U\subset M.$\,
It is well known that, under these conditions, the corresponding relative nullity distribution
$$x\in U\mapsto \Delta(x)\subset T_xM$$
is smooth and integrable, and its leaves are totally geodesic in
\,$M$\, and $Q_c^{n+p}.$\, Furthermore, if \,$\nu_0=\nu_{\min}$\,, these leaves  are
complete whenever \,$M$\, is complete.

In this setting, assuming that \,$\gamma:[0,a]\rightarrow M$\, is a
geodesic of \,$M$\, such that \,$\gamma([0,a))$\, is contained in a
leaf \,$\mathfrak{L}\subset U$\, of the
relative nullity distribution \,$\Delta,$\, we have that
\begin{itemize}[parsep=1ex]

\item[\small{\bf (P1)}] $\nu(\gamma(a))=\nu_0$\,.

\item[\small{\bf (P2)}] The first normal space \,$\mathscr{N}$\, is parallel at \,$\gamma(t)\, \forall t\in[0,a),$\, provided \,$\mathscr{N}$\, is
parallel at \,$\gamma(a).$

\end{itemize}

Property {\Small{\bf (P1)}}, together with the integrability of the relative nullity distribution,
is proved in \cite{ferus} (see also \cite{dajczer}--Chapter 5), and
property {\Small{\bf (P2)}} is  the content of Lemma 4 of \cite{rodriguez-tribuzy}.

The next result is due to D. Ferus, which he used in his proof of
Sacksteder Rigidity Theorem, as presented in \cite{dajczer}--Chapter 6.

\begin{proposition}[Ferus] \label{prop-incompleteleaves}
Let \,$f:M^n\rightarrow Q_c^{n+p}$\, be an isometric immersion, where \,$M^n$\, is assumed to be
compact for \,$c\le 0,$\, and complete otherwise. Assume further that,
in an open set \,$U\subset M,$\, \,$\nu= n-1.$\,
Then, no leaf of the relative nullity foliation is complete.
\end{proposition}

\begin{proof}[Proof of Theorem \ref{th-semiregular}]

Let us notice first that \,$M-M_c$\, is nonempty. Indeed, if it were not so,
we would have \,$\nu= n-1$\, in the
open set \,$M_c-M_{\rm tot}=M-M_{\rm tot}$\, (see Remark \ref{rem-nullity}).
In this case, \,$\Delta\subset T(M-M_{\rm tot})$\, would be  the minimum relative
nullity distribution and   its leaves would then be complete, since \,$M$\, is complete.
But, by Proposition \ref{prop-incompleteleaves}, this is impossible.

Next, we prove  that the first normal bundle
\,$\mathscr{N}$\, of \,$f$\, is parallel in \,$M-M_{\rm tot}$\,.
From Proposition \ref{prop-parallel}, \,$\mathscr{N}$\,  is parallel in the closure of \,$M-M_c$\,.
So, it remains to prove  that \,$\mathscr{N}$\, is parallel in \,$U={\rm int}(M_c-M_{\rm tot}),$\, which we can  assume is nonempty.
Given  \,$x\in U,$\, it follows from Proposition \ref{prop-incompleteleaves}
that any geodesic contained in a leaf \,$\mathfrak{L}\owns x,$\, with \,$x$\, as initial point,
must intersect the boundary of \,$U$\, at a point \,$y.$\, However, by property {\Small{\bf (P1)}},
\,$\nu(y)=n-1.$\, Hence, \,$y\not\in M_{\rm tot}$\,, which implies  that \,$y$\, is in the closure of \,$M-M_c$\,.
In particular, \,$\mathscr{N}$\, is parallel at \,$y.$\, This, together
with property {\Small{\bf (P2)}}, gives that \,$\mathscr{N}$\, is parallel at \,$x,$\,
and thus in all of \,$M-M_{\rm tot}$\,.

Since we are assuming that \,$M-M_{\rm tot}$\, is connected, it follows from
Proposition \ref{prop-reductioncodimension}  that
\,$f({\rm cl}(M-M_{\rm tot}))$\, is contained in a totally geodesic
submanifold \,$Q_c^{n+1}$\, of  \,$Q_c^{n+p}.$\,
In particular, for all \,$x\in\,{\rm cl}(M-M_{\rm tot}),$\, one has
\,$f_*(T_xM)\subset T_{f(x)} Q_c^{n+1}.$

If \,${\rm int}\, M_{\rm tot}$\, is nonempty, according to results
in \cite{obata}, the generalized Gauss map of \,$f$\, is constant on any connected component
\,$C$\, of \,$M_{\rm tot}$\,, which implies that \,$f(C)$\, is contained in
an $n$-dimensional totally geodesic submanifold \,$Q_c^n\subset Q_c^{n+p}$\,.
So, by continuity, for any \,$x\in\partial C,$\, one has
$$T_{f(x)}Q_c^n=f_*(T_xM)\subset T_{f(x)}Q_c^{n+1},$$
that is, \,$Q_c^n$\, is tangent to \,$Q_c^{n+1}$\, at \,$f(x).$\,
Therefore, \,$f(C)\subset Q_c^n\subset Q_c^{n+1},$\, from which we infer that
\,$f(M_{\rm tot})\subset Q_c^{n+1},$\, and thus that \,$f(M)\subset Q_c^{n+1}.$

Now, since we have reduced the codimension to one, statement (ii) follows directly from the hypothesis and
Sacksteder Rigidity Theorem.

Finally, in order to prove (i),
assume that \,${\rm Ric}_{\scriptscriptstyle{M}}\ge c.$\, Then, by the Bonnet-Myers Theorem,
\,$M$\, is compact for \,$c>0.$\,

Given a unit tangent field \,$X\in TM,$\, if
\,$\{X_1, X_2\,, \dots ,X_n\}\in TM$\, is an orthonormal tangent frame on \,$M$\,
such that \,$X_1=X,$\, it is well known that the equality
$$
{\rm Ric}_{\scriptscriptstyle{M}}(X)= c+\frac{1}{n-1}\langle\alpha_f(X,X),H\rangle-\frac{1}{n-1}\sum_{i=1}^{n}\|\alpha_f(X,X_i)\|^2
$$
holds.

So, if \,$H(x)=0,$\,  one has \,$\alpha_f(X,X_i)=0 \, (i=1,\dots ,n).$\, In particular,
\,$\alpha_f(X,X)=0,$\,
which implies that \,$x\in M_{\rm tot}$\,, since \,$X$\, is arbitrary.
Otherwise, \,$\langle\alpha_f(X,X),H\rangle$\, is nonnegative. Thus,
again from the fact that we are in  codimension one, we have that \,$\alpha_f$\, is semi-definite. This, together with
Jonker's Theorem (for \,$c=0$) and Theorem \ref{th-convexity} (for \,$c\ne 0$), yields (i) and finishes our proof.
\end{proof}

\end{document}